\documentclass[12pt,twoside,english,a4paper]{article}
\usepackage{babel,amssymb,amsthm}
\usepackage{graphicx}
\usepackage{amsmath}
\usepackage{color}

\newcommand{\jump}[1]{[\![#1]\!]}

\newcommand{\bff}[1]{{\mathbf{#1}}}
\newcommand{\bs}[1]{{\boldsymbol{#1}}}
\newcommand{\curl}{{\mathrm{curl}}}
\newcommand{\divv}{{\mathrm{div}}}

\newcommand{\HdivH}{\bff H^{-1/2}(\divv_\Gamma,\Gamma)}
\newcommand{\HcurlH}{\bff H^{-1/2}(\curl_\Gamma,\Gamma)}
\newcommand{\TD}{\mathrm{TD}}
\newcommand{\TC}{\mathrm{TC}} 
\newcommand{\HcurlG}{-1/2,\curl,\Gamma}
\newcommand{\HdivG}{-1/2,\divv,\Gamma}

\definecolor{red}{rgb}{0.8,0,0} 

\newtheorem{proposition}{Proposition}[section]
\newtheorem{corollary}[proposition]{Corollary}

\newtheorem{theorem}[proposition]{Theorem}

\numberwithin{equation}{section}

\title{New mapping properties of the \\ Time Domain Electric Field Integral Equation}
\date{\today}

\author{Tianyu Qiu \& Francisco--Javier Sayas\footnote{Department of Mathematical Sciences, University of Delaware, Newark DE 19716.  {\tt \{qty,fjsayas\}@udel.edu}. Partially funded by NSF (grant DMS 1216356)}}

\begin{document}

\maketitle

\begin{abstract}
We show some improved mapping properties of the Time Domain Electric Field Integral Equation and of its Galerkin semidiscretization in space. We relate the weak distributional framework with a stronger class of solutions using a group of strongly continuous operators.  The stability and error estimates we derive are sharper than those in the literature.
\end{abstract}

\section{Introduction}
The electric field integral equation (EFIE) has received much attention as a representation of electric fields that is conducive to efficient discretization schemes. The EFIE and its variants are a competitive alternative to the finite element discretization of the Maxwell equations due to two main advantages: (a) reducing problem dimension by one and, especially, (b) handling unbounded domains in a natural way. However, concerns over its numerical stability hamper the method's popularity in practical applications. Several papers, for example \cite{HiSc:2002,MeSm:2014,LiMoWe:2015,ChMo:2015}, address this issue with various techniques, yet most of them focus on the frequency domain analysis. Important as it is, the frequency domain equation is ill-suited to deal with broad band waves, which can be adequately treated in the time domain. Rigorous numerical study of the Time Domain Electric Field Integral Equation (TDEFIE) is, to our knowledge, scarce in the literature. Early attempts \cite{Terrasse:1993, BaBaSaVe:2013, ChMoWaWe:2012} all develop the kind of frequency domain analysis debuting in \cite{BaHa:1986a,BaHa:1986b} and then obtain the time domain estimates by inverse Laplace transforms or Plancherel identities. A lone contribution using time-domain techniques for the analysis ---not dealing with discretization, and limited to smooth domains--- is given in \cite{Rynne:1999}.

Pure time domain analysis has been shown to outperform the double-back-through-Laplace-domain approach in several situations. Time domain analysis in this context originated in \cite{Sayas:2013d} as a tool to analyze long term stability of several boundary integral formulations in acoustics, and was developed in \cite{DoSa:2013} to provide improved bounds for the retarded layer potentials and integral operators of transient acoustics. The same approach was further developed and refined for the direct integral equation formulation of transient scattering by a sound-soft obstacle \cite{BaLaSa:2015}, a boundary integral formulation for transmission problems in acoustics \cite{QiSa:2014}, and indirect formulations for Dirichlet and Neumann problems for acoustics \cite{Sayas:2014}. What distinguishes this paper from the previous time-domain analysis is the transformation of the abstract second order differential equation associated to a dissipative operator to a system of first order equations (in time as well as in space). This simplifies the analysis with respect to \cite{DoSa:2013, BaLaSa:2015, QiSa:2014} by avoiding the introduction of a cut-off boundary that was required to fit the problem in the correct functional framework. We note here that the techniques of  \cite{Sayas:2014} are not easy to adapt to the TDEFIE, due to non-minor complications in the associated Sobolev spaces.

The results in this paper improve the Laplace domain estimates in two ways: (a) less regularity is needed of the input data, i.e., the mapping properties are improved; (b) the constants in our estimates are independent of time, leading to a reassuring conclusion that the solution will not blow up as long as the input data is compactly supported. We also carry out the analysis meticulously to reveal the constant's dependence on the velocity of the wave, which might eventually shed some light on the low frequency breakdown of the EFIE.
Remarkably, these results are valid irrelevant of the scatterer's shape regularity, be it smooth or polyhedral, all thanks to the foundational work \cite{BuCi:1999,BuCoSh:2002}. They naturally incorporate as a special case the mapping properties of the boundary integral operators and layer potentials. We envision our results to be a unifying and instrumental step in carrying out the analysis of the full discretization for different time semidiscretization strategies like finite differences, space-time Galerkin or convolution quadrature. It is also surprising (although this is not new \cite{BaBaSaVe:2013, ChMoWaWe:2012}) to note that the TDEFIE is amenable to a general Galerkin discretization-in-space, and no discrete Hodge decompositions are needed, as opposed to the requirements of the frequency domain EFIE \cite{HiSc:2002}.

The paper is structured as follows. In Section 2, we write the time-dependent Maxwell equations into the TDEFIE and state the main results, concerning the regularity of the solution as a function of the data, before and after Galerkin semidiscretization-in-space, error estimates for Galerkin semidiscretization, and mapping properties of the forward operator.
In Section 3, we set up a mathematically rigorous background to understand the TDEFIE equation in the distributional sense.
In Section 4, we apply $C_0-$semigroup theory techniques to prove the main results, which are then compared with the existing results in the literature in Section 5.

\paragraph{Foreword on notation and background.}
Standard Sobolev space notations like $L^2(\mathcal O),H^m(\mathcal O)$ for boundary or domain $\mathcal O$ are assumed throughout the paper. Boldface notation, such as $\bff L^2(\mathcal O)$ and $\bff H^m(\mathcal O)$, is used for vector valued functions with each component in the corresponding scalar function space. 
Given an open set $\mathcal O$, $\|\cdot\|_{\mathcal O}$ denotes the norm of both $L^2(\mathcal O)$ and $\mathbf L^2(\mathcal O)$, while $(\cdot,\cdot)_{\mathcal O}$ is the associated inner product. 
We will write  $\mathcal B(X,Y)$ to denote the space of bounded linear operators between two Banach spaces $X$ and $Y$. 
For some very basic background on causal vector-valued distributions of a single variable and their Laplace transforms, we refer to \cite[Chapters 2 and 3]{Sayas:2014}, which reduces the scope of a vaste theory that can be explored in \cite{Treves:1967} or \cite{DaLi:1992}. The space of of $k$-times  continuously differentiable functions from an interval $I\subset\mathbb R$ to a Banach space $X$ is denoted
 $\mathcal C^k(I;X)$.

\section{Main results}\label{sec:2}

Let $\Omega_-\subset\mathbb R^3$ be a bounded Lipschitz open set with boundary $\Gamma$, let $\Omega_+:=\mathbb R^3\setminus\overline{\Omega_-}$ be its exterior, and let $\bs\nu:\Gamma \to \mathbb R^3$ be the outward pointing unit normal vector field on $\Gamma$. Our goal is the study of the properties of a boundary integral formulation for the Maxwell equations  in vacuum, written in terms of the electric field, in the exterior of a perfectly conducting scatterer occupying $\Omega_-$. In PDE notation, the problem can be written as the search for  $\bff E:\Omega_+\times [0,\infty) \to \mathbb R^3$ (the scattered electric field) such that
\begin{subequations}\label{eq:2.3}
\begin{alignat}{6}
c^{-2}\partial_t^2 \bff E+\nabla\times\nabla\times \bff E = \bff 0 & \qquad && \mbox{in $\Omega_+\times [0,\infty)$},\\
(\bff E+\bff E^{\mathrm{inc}})\times\bs\nu=\bff 0 & & & \mbox{in $\Gamma\times [0,\infty)$},\\
\bff E(\cdot,0)=\mathbf 0 & & & \mbox{in $\Omega_+$},\\
\partial_t \bff E(\cdot,0)=\mathbf 0 & & & \mbox{in $\Omega_+$}.
\end{alignat}
\end{subequations}
Here $\bff E^{\mathrm{inc}}$ is a given incident field and the curl operator is taken on the space variables. The scattered field is assumed to be radiating, i.e., $\mathbf E(\cdot,t)$ has a bounded spatial support for all $t$.
We will use a single layer representation of $\bff E$ in terms of an unknown $\mathbf J:\Gamma \times \mathbb R\to \mathbb R^3$ satisfying
\begin{equation}
\mathbf J(\cdot,t)\equiv \bff 0 \quad \forall t <0, \qquad \mathbf J(\cdot,t)\times \bs\nu =\bff 0 \qquad \forall t, 
\end{equation} 
that is, $\mathbf J$ is a causal tangential vector field on $\Gamma$. The single layer potential ansatz is given by the formula \cite[2.2.3]{Costabel:2004}
\begin{eqnarray}
\bff E(\mathbf x,t)
	&\!\!=\!\! &-c^{-1}\int_\Gamma \frac1{4\pi|\mathbf x-\mathbf y|}\partial_t \mathbf J(\mathbf y,t-c^{-1}|\bff x-\bff y|)\mathrm d\Gamma(\mathbf y)\\
\nonumber
	&&+c\nabla \int_\Gamma \frac1{4\pi|\mathbf x-\mathbf y|}
		\left(\int_0^{t-c^{-1}|\bff x-\bff y|}\!\!\!\!\!\!\!\!\mathrm{div}_\Gamma \mathbf J(\bff y,\tau)\mathrm d\tau\right)
		\mathrm d\Gamma(\bff y),\quad (\mathbf x,t)\in \Omega_+\times [0,\infty),
\end{eqnarray}
where $\mathrm{div}_\Gamma$ is the tangential divergence operator, and $\nabla$ is the gradient in the $\bff x$ variable. The tangential component of the Maxwell single layer potential on points $\mathbf x\in \Gamma$ is 
\begin{eqnarray*}
(\bs{\mathcal V}_c*\mathbf J)(\bff x,t) 
	&:=& -c^{-1} \bs\nu(\bff x)\times \left(\int_\Gamma \frac1{4\pi|\bff x-\bff y|}\partial_t 
		\mathbf J(\mathbf y,t-c^{-1}|\bff x-\bff y|)\mathrm d\Gamma(\mathbf y) \right)\times \bs\nu(\bff x)\\
	& & + c \left(\nabla_\Gamma \int_\Gamma \frac1{4\pi|\mathbf x-\mathbf y|}
		\left(\int_0^{t-c^{-1}|\bff x-\bff y|}\!\!\!\!\!\!\!\!\mathrm{div}_\Gamma \mathbf J(\bff y,\tau)\mathrm d\tau\right)
		\mathrm d\Gamma(\bff y) \right),
\end{eqnarray*}
where $\nabla_\Gamma$ is the tangential gradient. We emphasize that so far our presentation is merely formal and the convolution symbol in the definition of $\bs{\mathcal V}_c*\mathbf J$ is just notation. The single layer potential will be denoted $\mathbf E=\bs{\mathcal S}_c*\bff J$ and it is defined on both sides of the boundary. The tangential component of the single layer potential is continuous accross the boundary, namely
\[
(\bff E^--\bff E^+)\times \bs\nu=\bff 0,
\]
where the $\pm$ superscripts denote limits from $\Omega^\pm$.
The time derivative of the density can be recovered from tangential values of the single layer potential with the formula
\[
\partial_t \bff J=c\, (\nabla\times\bff E^+-\nabla\times \bf E^-)\times \bs\nu.
\]
More details on the meaning of tangential trace operators will be given below. The only reading of the incident field on the boundary of the conductor is given by its  tangential component:
\[
\bs\beta(\bff x,t):=-\bs\nu(\bff x) \times \bff E^{\mathrm{inc}}(\bff x,t)|_\Gamma \times \bs \nu(\bff x).
\]
The Time Domain Electric Field Integral Equation (TDEFIE) determines $\mathbf J$ by imposing the boundary condition
\begin{equation}\label{eq:2.4}
(\bs{\mathcal V}_c*\mathbf J)(\bff x,t)=\bs\beta(\bff x,t), \qquad (\mathbf x,t)\in \Gamma\times [0,\infty).
\end{equation}
The solvability analysis for this equation will be carried out at the same time as the analysis of a semidiscrete-in-space version of it. Galerkin semidiscretization starts with a finite dimensional space $\bs X_h$ containing elements $\bs\mu^h$ satisfying:
\[
\bs\mu^h\in \bff L^2(\Gamma), 
\qquad 
\bs\mu^h\times\bs\nu\equiv \bff 0, 
\qquad
\mathrm{div}_\Gamma \bs\mu^h\in L^2(\Gamma).
\]
For instance, if $\Gamma$ is a polyhedron, the Raviart-Thomas or Rao-Wilton-Glisson \cite{RaWiGl:1982} elements defined on a triangulation of $\Gamma$ can be used as a discrete space. The semidiscrete Galerkin equations look for $\mathbf J^h:\Gamma\times \mathbb R\to\mathbb R^3$ satisfying
\begin{equation}
\mathbf J^h(\cdot,t)\equiv \bff 0 \quad t<0, \qquad \mathbf J^h (\cdot,t)\in \bs X_h, \quad t\ge 0,
\end{equation}
and 
\begin{equation}\label{eq:2.6}
\int_\Gamma (\bs{\mathcal V}_c*\mathbf J^h)(\mathbf x,t)\cdot\bs\mu^h(\bff x)\mathrm d\Gamma(\bff x)
=\int_\Gamma\bs\beta(\bff x,t) \cdot\bs\mu^h(\bff x,t)\mathrm d\Gamma(\bff x)
\quad \forall\bs \mu^h\in \bs X_h, 
\,\forall t.
\end{equation}
The approximated electric field is the result of inputting the density $\bff J^h$ in the single layer potential expression $\mathbf E^h=\bs{\mathcal S}_c*\mathbf J^h$.

\paragraph{Weak tangential traces.}
In order to state the main result of this paper we need some Sobolev space notation. For an introduction to Sobolev spaces
related to the Maxwell equation we refer to \cite{Monk:2003}. Trace theorems, and a full characterization of the range of the trace operators, are studied in \cite{BuCi:1999,BuCi:2001} for polyhedra, and in \cite{BuCoSh:2002} for general Lipschitz domains. We briefly recall some definitions and results. In the space
\[
\bff H(\mathrm{curl},\Omega_+):=\{\bff u\in \bff L^2(\Omega_+)\,:\, \nabla \times \bff u \in \bff L^2(\Omega_+)\},	
\]
endowed with its natural norm
\[
\| \bff u\|_{\mathrm{curl},\Omega_+}^2:=
\| \bff u\|_{\Omega_+}^2 + \| \nabla\times \bff u\|_{\Omega_+}^2,
\]
we can define a tangential trace operator that extends $\gamma_\tau^+ \bff u= \bff u\times \bs\nu$. This operator is bounded and surjective from $\bff H(\mathrm{curl},\Omega_+)$ to $\bff H^{-1/2}(\mathrm{div}_\Gamma,\Gamma)$, which, roughly speaking, is the space of $\bff H^{-1/2}(\Gamma)$ tangential vector fields whose surface divergence is in $H^{-1/2}(\Gamma)$. The space $\bff H^{-1/2}(\divv_\Gamma,\Gamma)$ is endowed with the norm
\[
\| \bs \rho \|_{-1/2,\divv,\Gamma}^2:=\| \bs \rho\|_{-1/2,\Gamma}^2
+\| \divv_\Gamma \bs\rho\|_{-1/2,\Gamma}^2.
\]
In $\bff H(\mathrm{curl},\Omega_+)$, we can define the tangential boundary component that extends $\pi_\tau^+\bff u=\bs\nu\times \bff u\times \bs\nu$ and is bounded and surjective onto $\bff H^{-1/2}(\mathrm{curl}_\Gamma,\Gamma)$, where $\curl_\Gamma$ is the surface curl. Two  traces, $\gamma_\tau^-$ and $\pi_\tau^-$, can also be defined from the interior domain. For functions $\bff u,\bff v$ that are smooth enough on both sides of $\Gamma$ with bounded support, we have the integration-by-parts formula
\[
(\bff u,\nabla\times\bff v)_{\mathbb R^3\setminus\Gamma}
-(\nabla\times \bff u,\bff v)_{\mathbb R^3\setminus\Gamma}=
\langle \gamma_\tau^-\bff u,\pi_\tau^-\bff v\rangle-
\langle \gamma_\tau^+\bff u,\pi_\tau^+\bff v\rangle,
\]
where $(\cdot,\cdot)_{\mathbb R^3\setminus\Gamma}$ is the  $\bff L^2(\mathbb R^3\setminus\Gamma)$ inner product and the angled bracket is the $\bff L^2(\Gamma)$ inner product. This formula can be extended to $\bff u,\bff v\in \bff H(\mathrm{curl},\mathbb R^3\setminus\Gamma)$, with the angled bracket $\langle\cdot,\cdot\rangle$ now denoting the duality product of $\bff H^{-1/2}(\mathrm{div}_\Gamma,\Gamma)$ and $\bff H^{-1/2}(\mathrm{curl}_\Gamma,\Gamma)$. As is well understood since \cite{BuCi:1999,BuCi:2001,BuCoSh:2002}, these two spaces are dual to each other with the duality product extending the inner product of  square integrable tangential vector fields. The rotation operator $\bs\xi\mapsto \bs\xi\times\bs\nu$, acting on square integrable  tangential vector fields can be extended to an isometric isomorphism between $\bff H^{-1/2}(\divv_\Gamma,\Gamma)$ and $\bff H^{-1/2}(\curl_\Gamma,\Gamma)$. Of capital importance will be the following jump operator, defined in $\bff H(\curl^2,\mathbb R^3\setminus\Gamma)=\{ \bff u\in \bff H(\curl,\mathbb R^3\setminus\Gamma)\,:\, \nabla \times \nabla \times \bff u\in \bff L^2(\mathbb R^3\setminus\Gamma)\}$:
\begin{equation}
\jump{\bff u}_N:=\gamma_\tau^- \nabla \times \bff u-\gamma^+_\tau \nabla \times \bff u.	
\end{equation}

\paragraph{Dependence with respect to time.}
Let us set up the kind of notation we will henceforth use for functions of the space and time variables. We will assume that functions are defined as $f:\mathbb R\to X$, where $X$ is a Sobolev space on a domain or on $\Gamma$. The time derivative of $f$ will be denoted $\dot f$. We will also use the antidifferentiation symbol
\[
\partial^{-1} f(t):=\int_0^t  f(\tau)\mathrm d\tau,
\]
where we are using Bochner integration in the corresponding space $X$. A key space will be
\[
\mathcal W^k_+(X):=\{ f:\mathbb R\to X\,:\, f\in \mathcal C^{k-1}(\mathbb R;X), f\equiv 0 \mbox{ in $(-\infty,0)$},
f^{(k)} \in L^1(\mathbb R;X)\}, 
\]
where the highest order differentiation is taken in the sense of $X$-valued distributions over $\mathbb R$. In this space we can define the cummulative seminorms
\begin{equation}\label{eq:cumm}
H_k(f,t\,|\, X):=\sum_{\ell=0}^k \int_0^t \| f^{(\ell)}(\tau)\|_X\,\mathrm d\tau, \qquad t\ge 0.
\end{equation}
Note that $f\in \mathcal W^k_+(X)$ for $k\ge 1$, then $\| f^{(k-1)}(t)\|$ is uniformly bounded.
Another key space will be 
\[
\mathcal C^k_+(X):=\{ f\in \mathcal C^k(\mathbb R;X)\,:\,  \, f\equiv 0 \mbox{ in $(-\infty,0)$}\}.
\]

The main theorems of this paper concern the dependence of the solution of \eqref{eq:2.6} and of $\mathbf E^h=\bs{\mathcal S}_c*\mathbf J^h$ with respect to $\bs\beta:=-\bs\nu\times \mathbf E^{\mathrm{inc}}|_\Gamma\times \bs\nu$, as well as the error of the approximations of $\mathbf J^h$ to $\mathbf J$ and $\mathbf E^h$ to $\mathbf E$. The theorems will only appeal to the fact that it is a closed subspace of $\bff H^{-1/2}(\mathrm{div}_\Gamma,\Gamma)$ and will not break down even when   $\bs X_h$ is infinite dimensional. In particular the first theorem can also be read as a mapping estimate for the inversion of the continuous equation \eqref{eq:2.4}, which we rephrase as Corollary \ref{cor:2.3}. The expression {\em independent of $h$} has to be understood as independent of the choice of $\bs X_h$.

\begin{theorem}\label{th:1}
Let $\bs \beta\in \mathcal W^2_+(\bff H^{-1/2}(\curl_\Gamma,\Gamma))$, let $\mathbf J^h$ be the solution of \eqref{eq:2.6} and $\bff E^h=\bs{\mathcal S}_c*\bff J^h$. Then, 
$\bff E^h \in \mathcal C^1_+(\bff L^2(\Omega_+))\cap \mathcal C^0_+(\bff H(\curl,\Omega_+))$,
$\bff J^h \in \mathcal C^0_+(\bff H^{-1/2}(\divv_\Gamma,\Gamma))$, 
and there exists $C>0$, independent of $h$ and $t$, such that
\[
\| \bff J^h(t)\|_{\HdivG} + \| \bff E^h(t)\|_{\curl,\Omega_+} \leq 
C \max\{c,c^{-2}\}\, H_2(\bs\beta, t \,|\, \HcurlH)
\]
for all $t\geq 0$.
\end{theorem}

\begin{theorem}\label{th:2}
Let $\bff J$ and $\bff J^h$ be the respective solutions of \eqref{eq:2.4} and \eqref{eq:2.6} and let $\bff E=\bs{\mathcal S}_c*\bff J$ and $\bff E^h=\bs{\mathcal S}_c*\bff J^h$. Let $\Pi_h:\bff H^{-1/2}(\divv_\Gamma,\Gamma)\to \bs X_h$ be the orthogonal projection onto $\bs X_h$. 
If $\bff J\in \mathcal W^2_+(\bff H^{-1/2}(\divv_\Gamma,\Gamma))$, then there exists $C>0$, independent of $h$ and $t$, such that
\begin{align*}
\| \bff J(t)-\bff J^h(t)\|_{-1/2,\divv,\Gamma} + \| \bff E(t)-\bff E^h(t) & \|_{\curl,\Omega_+} \\
	&\le C\max\{c,c^{-2}\}\, H_2(\bff J-\Pi_h \bff J,t\,|\,\bff H^{-1/2}(\divv_\Gamma,\Gamma))
\end{align*}
for all $t\geq 0$.
\end{theorem}

\begin{corollary}\label{cor:2.3}
Let $\bs \beta\in \mathcal W^2_+(\bff H^{-1/2}(\curl_\Gamma,\Gamma))$ and $\bff J$ be the unique solution of $\bs{\mathcal V}_c*\bff J=\bs\beta$. Then $\bff J\in \mathcal C^0_+(\HdivH)$, $\bs{\mathcal S}_c*\bff J\in  \mathcal C^1_+(\bff L^2(\Omega_+))\cap \mathcal C^0_+(\bff H(\curl,\Omega_+))$, and there exists $C>0$, independent of $t$, such that
\[
\| \bff J(t)\|_{\HdivG} + \| (\bs{\mathcal S}_c*\bff J)(t)\|_{\curl,\Omega_+}
\le C \max\{c,c^{-2}\}\, H_2(\bs\beta, t \,|\, \HcurlH)
\]
for all $t\ge 0$.
\end{corollary}

\begin{corollary}\label{cor:2.4}
Let  $\bff J\in \mathcal W^2_+(\bff H^{-1/2}(\divv_\Gamma,\Gamma))$. Then $\bs{\mathcal S_c}*\bff J\in \mathcal C^1_+(\bff L^2(\Omega_+))\cap\mathcal C^0_+(\bff H(\curl,\Omega_+))$, and therefore $\bs{\mathcal V}_c*\bff J\in \mathcal C^0_+(\HcurlH))$, and there exists $C>0$, independent of $t$, such that
\[
\|(\bs{\mathcal V}_c*\bff J)(t)\|_{\HcurlG}+ \| (\bs{\mathcal S}_c*\bff J)(t)\|_{\curl,\Omega_+}
\le C\max\{c,c^{-2}\}\, H_2(\bff J,t\,|\,\bff H^{-1/2}(\divv_\Gamma,\Gamma))
\]
for all $t\ge 0$.
\end{corollary}

\section{Distributional potentials}\label{sec:3}

In this section we pave the way for the proofs of Theorems \ref{th:1} and \ref{th:2} by giving a very precise mathematical description of the problems: (a) given $\boldsymbol\beta$, compute $\mathbf E^h$ and $\mathbf J^h$, and (b) given $\mathbf J$, compute $\mathbf E^h-\mathbf E$ and $\mathbf J^h-\mathbf J$. We will handle both problems simultaneously. The theory is going to be developed with considerable latitude in the choice of the space $\bs X_h$. From now on, we will only assume that $\bs X_h$ is a  closed subspace of $\HdivH$. The two limiting cases $\bs X_h=\HdivH$ and $\bs X_h=\{ \bff 0\}$ will be discussed at the end of this section. We will use the polar set of $\bs X_h$
\[
\bs X_h^\circ:=\{ \bs\eta \in \HcurlH\,:\, \langle \bs\xi,\bs\eta\rangle=0 \quad \forall \bs\xi\in \bs X_h\}.
\]
Since $\bs X_h$ and $\bs X_h^\circ$ are closed, we can define bounded operators $\mathrm P_{\bs X_h}:\HdivH\to \HdivH$ and $\mathrm P_{\bs X_h^\circ}:\HcurlH\to \HcurlH$ such that
\begin{equation}\label{eq:3.0}
\bs\xi\in \bs X_h \quad \Longleftrightarrow \quad \mathrm P_{\bs X_h}\bs\xi=\bff 0
\qquad
\mbox{and}
\qquad
\bs\eta\in \bs X_h^\circ \quad \Longleftrightarrow \quad \mathrm P_{\bs X_h^\circ}\bs\eta=\bff 0.
\end{equation}
What these projections are is unimportant. The orthogonal projections on the orthogonal complement of the respective spaces can be taken for this role.

The transient single layer potential for Maxwell's equation can be rigorously defined using a Laplace transform. The techniques are well-known and available in the literature. We will only introduce the definitions essential for a correct handling of the potential and its tangential component on the boundary. The following theorem lays the framework for causal vector-valued distributions, relating some distributions to their Laplace transforms.

\begin{theorem}\label{th:FJS}
{\rm \cite[Chapter 3]{Sayas:2014}} Let $X$ be a Banach space and let $f$ be an $X$-valued distribution in $\mathbb R$. The following statement on $f$
\begin{quote}
there exists a continuous function $g:\mathbb R\to X$ such that $g(t)=0$ for all $t\le 0$ and such that $\| g(t)\| \le C t^m$ for all $t\ge 1$ with $m\ge 0$, and there exists a non-negative integer $k$ such that $f=g^{(k)}$
\end{quote}
is equivalent to
\begin{quote}
$f$ admits a Laplace transform $\mathrm F=\mathcal L\{ f\}$ defined in $\mathbb C_+:=\{ s\in \mathbb C\,:\, \mathrm{Re}\,s>0\}$ and satisfying $\| \mathrm F(s)\|\le C_{\mathrm F}(\mathrm{Re}\,s) |s|^\mu$ for all $s\in \mathbb C_+$, where $\mu \in \mathbb R$ and $C_{\mathrm F}:(0,\infty)\to (0,\infty)$ is non-increasing and satisfies $C_{\mathrm F}(\sigma)\le C \sigma^{-\ell}$ for all $\sigma<1$ for some $C>0$ and $\ell \ge 0$.
\end{quote}
\end{theorem}

Following \cite{Sayas:2014}, the set of all causal distributions characterized by Theorem \ref{th:FJS} will be denoted $\TD(X)$ (TD as in time-domain). Note that if $f\in \TD(X)$ and $A\in \mathcal B(X,Y)$, then $Af\in \TD(Y)$. Note also that for $\mathcal W^k_+(X)\subset \TD(X)$ for all $k\ge 0$.

\paragraph{The Maxwell single layer potential.}
For brevity, we name the space
\[
\bff M:=\mathbf H(\curl,\mathbb R^3)\cap \bff H(\curl^2,\mathbb R^3\setminus\Gamma).
\]
We endow $\bff M$  with the norm
\[
\|\bff u\|_{\bff M}^2:=
	\|\bff u\|_{\mathbb R^3}^2
	+\|\nabla\times\bff u\|_{\mathbb R^3}^2
	+\|\nabla\times\nabla\times\bff u\|_{\mathbb R^3\setminus\Gamma}^2,
\]
and note that $\pi_\tau:\bff M \to \HcurlH$ and $\jump{\cdot}_N:\bff M \to \HdivH$ are bounded.
Given $\bs\xi \in \mathrm{TD}(\HdivH)$, we can find a unique
$\bff u\in \mathrm{TD}(\bff M)$
 such that
\begin{equation}\label{eq:3.1}
c^{-2}  \ddot{\bff u}+\nabla\times\nabla\times \bff u = \bff 0,
\qquad
\jump{\bff u}_N=-c^{-1}\dot{\bs\xi}.
\end{equation}
The first equation in \eqref{eq:3.1} has to be understood in the sense of $\bff L^2(\mathbb R^3\setminus\Gamma)$-valued distributions, while the second one is an equation in the sense of $\HdivH$-valued distributions. Implicit to the fact that $\bff u$ is $\bff M$-valued is the equality $\pi_\tau^-\bff u=\pi_\tau^+\bff u$. The operator that given $\bs\xi$ outputs $\bff u$ is a convolution operator and we denote it as $\bff u=\bs{\mathcal S}_c*\bs\xi$. We then denote $\bs{\mathcal V}_c*\bs\xi:=\pi_\tau \bs{\mathcal S}_c*\bs\xi$, which is a causal $\HcurlH$-valued distribution. Note that the convolution product of causal distributions is always well defined, and therefore, the convolutions $\bs{\mathcal S}_c*\bs\xi$ and $\bs{\mathcal V}_c*\bs\xi$ can be extended to arbitrary causal $\HdivH$-valued distributions $\bs\xi$. However, the relation of these operators to the transmission problem \eqref{eq:3.1} is unclear when $\bs\xi$ does not admit a Laplace transform. The existence and uniqueness of solution of \eqref{eq:3.1} is proved using the Laplace transform (see \cite{BaBaSaVe:2013} for a recent and careful exposition). We now sketch the idea with wave speed $c=1$.
Given $\bs\eta\in \HdivH$ we look for $\bff U\in \bff M$ satisfying
\begin{equation}\label{eq:3.3A}
s^2 \bff U+\nabla\times\nabla \times\bff  U=\bff 0, \qquad \jump{\bff U}_N=-s \bs\eta
\end{equation}
or, equivalently, we look for $\bff U\in \bff H(\curl,\mathbb R^3)$ satisfying
\begin{equation}\label{eq:3.3B}
s^2 (\bff U,\mathbf v)_{\mathbb R^3}+(\nabla \times \bff U,\nabla\bff v)_{\mathbb R^3}
	=-s\langle \bs\eta,\pi_\tau \bff v\rangle 
 	\qquad \forall \bff v \in \bff H(\curl,\mathbb R^3).
\end{equation}
The solution operator for \eqref{eq:3.3A} is a bounded linear operator $\mathrm S(s):\HdivH \to \bff M$. Proving a bound for $\|\mathrm S(s)\|$ in the style of the Laplace bounds of Theorem \ref{th:FJS}, we can show the existence of $\bs{\mathcal S}_c\in \TD(\mathcal B(\HdivH,\bff M))$ such that $\mathrm S(\cdot/c)=\mathcal L\{ \bs{\mathcal S_c}\}$. Then $\bs{\mathcal S}_c*\bs\xi$ is characterized by its Laplace transform $\mathrm S(s/c) \mathcal L\{\bs\xi\}(s)$. 

\paragraph{A transmission problem.} Let $\bs\beta\in \TD(\HcurlH)$ and $\bs\xi\in \TD(\HdivH)$. We look for $\bff u\in \TD(\bff M)$ satisfying
\begin{subequations}\label{eq:TP}
\begin{alignat}{6}
\label{eq:TPa}
	c^{-2} \ddot{\bff u}+\nabla\times\nabla\times \bff u = \bff 0, \\
\label{eq:TPb}
	\pi_\tau \bff u-\bs\beta \in \bs X_h^\circ,\\
\label{eq:TPc}
	\jump{\bff u}_N+c^{-1}\dot{\bs\xi} \in \bs X_h.
\end{alignat}
\end{subequations}
Let us first clarify where these equations take place. Equation \eqref{eq:TPa} is an equality of $\bff L^2(\mathbb R^3\setminus\Gamma)$-valued distributions. Equations \eqref{eq:TPb} and \eqref{eq:TPc} can be understood as equalities in the sense of distributions with values in $\HcurlH$ and $\HdivH$ respectively. In other words, we write the equivalent equations using the projectors \eqref{eq:3.0}:
\[
\mathrm P_{\bs X_h^\circ}(\pi_\tau \bff u-\bs\beta)=\bff 0,
\qquad
\mathrm P_{\bs X_h}(\jump{\bff u}_N-c^{-1}\dot{\bs\xi})=\bff 0.
\]
We remark that these are the equations for distributions of the time variable. An informal way of understanding it would be to assume that all quantities are functions of $t$ and the equations are satisfied for all $t$.

\begin{proposition}
For any $\bs\beta\in \TD(\HcurlH)$ and $\bs\xi\in \TD(\HdivH)$, problem \eqref{eq:TP} has a unique solution $\bff u\in \TD(\bff M)$.
\end{proposition}

\begin{proof}
The proof follows from a simple combination of arguments in \cite{LaSa:2009b} and \cite{BaBaSaVe:2013}. We just sketch the main steps. We let $\mathrm B=\mathcal L\{ \boldsymbol\beta\}$ and $\Xi=\mathcal L\{\boldsymbol\xi\}$.  For $s\in \mathbb C_+$ we solve the coercive variational problem
\begin{subequations}\label{eq:3.4}
\begin{alignat}{6}
	&\bff U(s) \in \bff H(\curl,\mathbb R^3),\qquad \pi_\tau \bff U(s)-\mathrm B(s)\in \bs X_h^\circ,\\
	&(\nabla \times \bff U(s),\nabla \times \bff v)_{\mathbb R^3} 
		+(s/c)^2 (\bff U(s),\bff v)_{\mathbb R^3} 
		=-(s/c)\langle \Xi(s),\pi_\tau \bff v\rangle \quad \forall \bff v \in \bff H_0, 
\end{alignat}
\end{subequations}
where $\bff H_0:=\{ \bff v\in \bff H(\curl,\mathbb R^3)\,:\, \pi_\tau \bff v\in \bs X_h^\circ\}$. The solution of \eqref{eq:3.4}, parametrized in the variable for the Laplace transform $s$, is the Laplace transform of an $\bff M$-valued distribution $\bff u$ that solves \eqref{eq:TP}. Following techniques in \cite{BaBaSaVe:2013} it is simple to prove that
\[
|\tfrac{s}{c}|\| \bff U(s)\|_{\curl,\mathbb R^3}
	+\|\nabla\times\nabla\times \bff U(s)\|_{\mathbb R^3\setminus\Gamma}
	\le C(\mathrm{Re}(\tfrac{s}{c})) |\tfrac{s}{c}|^3(\|\Xi(s)\|_{-1/2,\divv,\Gamma}+\|\mathrm B(s)\|_{-1/2,\curl,\Gamma}),
\] 
where $C(\sigma)=C / (\sigma \min\{1,\sigma^2\})$ and $C$ depends only on $\Gamma$. Theorem \ref{th:FJS} can then be invoked to prove that $\bff U=\mathcal L\{\bff u\}$ where $\mathbf u\in \TD(\bff M)$ solves \eqref{eq:TP}.
\end{proof} 

The solution of \eqref{eq:3.4} can be written as
$
\bff U(s) = \mathrm G_h(s) \mathrm B(s)+\mathrm E_h(s) \Xi(s),
$
using two bounded operators $\mathrm G_h(s) \in \mathcal B(\HcurlH,\bff M)$, $\mathrm E_h(s)\in \mathcal B(\HdivH,\bff M)$. Taking the inverse Laplace transform, this formula becomes the sum of two convolutions
\begin{equation}\label{eq:3.7}
\bff u=\mathcal G_h*\bs\beta+\mathcal E_h*\bs\xi,
\end{equation}
where $\mathcal G_h\in \TD(\mathcal B(\HcurlH,\bff M))$ and $\mathcal E_h\in \TD(\mathcal B(\HdivH,\bff M)).$ By uniqueness of solution to \eqref{eq:TP}, if $\bs\xi^h\in \TD(\HdivH)$ satisfies $\bs\xi^h\in \bs X_h$, then $\mathcal E_h*\bs\xi^h=\bff 0$. Therefore
\begin{equation}\label{eq:3.8}
\mathcal E_h*\bs\xi=\mathcal E_h * (\bs\xi-\Pi_h \bs\xi).
\end{equation}
The next two results express the solution to the problems of Section \ref{sec:2} using the  convolution operators in \eqref{eq:3.7}. In particular $\bs\beta\mapsto\mathcal G_h*\bs\beta$ will be the semidiscrete EFIE solution operator (Theorem \ref{th:1}), while $\bs\xi\mapsto\mathcal E_h * \bs\xi$ will be related to the associated error operator (Theorem \ref{th:2}).

\begin{proposition}\label{prop:3.3}
Let $\bs\beta \in \TD(\HcurlH)$ and define
\[
\bff E^h=\mathcal G_h*\boldsymbol\beta,
\qquad
\bff J^h=-c\,\partial^{-1}\jump{\bff E^h}_N.
\] 
Then $\bff J^h\in \TD(\HdivH)$ and $\bff E^h\in \TD(\bff M)$ are characterized by
\begin{equation}\label{eq:3.5}
\bff J^h \in \bs X_h, 
\qquad
\bs{\mathcal V}_c*\bff J^h-\bs\beta\in \bs X_h^\circ,
\qquad
\bff E^h=\bs{\mathcal S}_c*\bff J^h.
\end{equation}
\end{proposition}

\begin{proof}
This is a straightforward consequence of the definition of $\mathcal G_h$ and of the characterization of the Maxwell single layer potential by the transmission problem \eqref{eq:3.1}. Note that the second equation in \eqref{eq:3.5} can be equivalently written as a Galerkin semidiscrete equation: $\langle \bs\mu^h,\bs{\mathcal V}_c*\bff J^h \rangle=\langle \bs\mu^h,\bs\beta\rangle$ for all $\bs\mu^h \in \bs X_h.$
\end{proof}

\begin{proposition}\label{prop:3.4}
Let $\bff J\in \TD(\HdivH)$ and define
\[
\bff E^h=\bs{\mathcal S}_c*\bff J-\mathcal E_h*\bff J,
\qquad
\bff J^h=-c\,\partial^{-1}\jump{\bff E^h}_N.
\]
Then $\bff J^h\in \TD(\HdivH)$ and $\bff E^h\in \TD(\bff M)$ are characterized by
\begin{equation} \label{eq:3.10}
\bff J^h \in \bs X_h, 
\qquad 
\bs{\mathcal V}_c*(\bff J^h-\bff J)\in \bs X_h^\circ,
\qquad
\bff E^h=\bs{\mathcal S}_c*\bff J^h.
\end{equation}
Therefore, if $\bff E=\bs{\mathcal S}_c*\bff J$, then
\begin{equation}
\bff E-\bff E^h=\mathcal E_h*\bff J=\mathcal E_h*(\bff J-\Pi_h \bff J),
\qquad
\bff J-\bff J^h=-c\,\partial^{-1}\jump{\bff E-\bff E^h}_N.
\end{equation}
\end{proposition}

\begin{proof}
This result recasts Galerkin semidiscrete equations from the point of view of the exact solution: $\langle\bs\mu^h,\bs{\mathcal V}_c*\bff J^h\rangle=\langle\bs\mu^h,\bs{\mathcal V}_c*\bff J\rangle$ for all $\bs\mu^h \in \bs X_h.$ Note that \eqref{eq:3.8} allows us to substract $\Pi_h\bff J$ in the argument of the operator of convolution by $\mathcal E_h$.
\end{proof}

\paragraph{Remark.} The convolution $\mathcal G_h*\bs\beta+\mathcal E_h*\bs\xi$ is well defined for any causal pair of distributions $(\bs\beta,\bs\xi)$ (by causal, we mean that their support is contained in $[0,\infty)$). This is due to the fact that the operator-valued distributions $\mathcal G_h$ and $\mathcal E_h$ are themselves causal. Moreover, $\bff u=\mathcal G_h*\bs\beta+\mathcal E_h*\bs\xi$ is a solution of \eqref{eq:TP}, because we can understand each of the equations in \eqref{eq:TP} as the result of applying operators to the distributions $\mathcal G_h$ and $\mathcal E_h$. What is not guaranteed is the uniqueness of solution of \eqref{eq:TP} unless we restrict the space of possible solutions. We have opted for the set $\TD(\bff M)$, which is somewhat restrictive but large enough for our purposes. Uniqueness can also be asserted in more general subspaces of the space of causal distributions, defined by the existence of a Laplace transform and some bounds on its behavior.

\paragraph{Two particular cases.} Assume that we take $\bs X_h=\HdivH$ so that $\bs X_h^\circ=\{ \bff 0\}$. Then $\mathcal E_h=0$ and we are solving the problem
\[
\bs{\mathcal V}_c*\bff J=\bs\beta, \qquad \bff u=\bs{\mathcal S}_c*\bff J,
\]
i.e., we are dealing with the non-discretized inverse, also known as the continuous stability estimate. If we take $\bs X_h=\{\bff 0\}$ instead, then $\bs X_h^\circ=\HcurlH$. In this case $\mathcal G_h=0$, $\mathcal E_h=\bs{\mathcal S}_c$, and we are dealing with
\[
\bff u=\bs{\mathcal S}_c*\bs\xi,
\]
i.e., we are handling the single layer potential.

\section{Proofs of the main results}

\paragraph{A first order system.}
Given $\bs\beta\in \TD(\HcurlH)$ and $\bs\xi\in \TD(\HdivH)$, we look for $\bff u\in \mathcal \TD (\bff M)$ and $\bff v\in \TD (\bff H(\curl;\mathbb{R}^3\backslash\Gamma))$ satisfying
\begin{subequations}\label{eqn:weak}
\begin{alignat}{6}
& \dot{\bff u} - c\nabla \times \bff v=\bff 0,&\qquad&
& \dot{\bff v} + c\nabla \times \bff u=\bff 0,\\
& \pi_\tau \bff u-\bs\beta \in \bs X_h^\circ, &\qquad &
& \jump{\gamma_\tau \bff v} - \bs \xi \in \bs X_h.
\end{alignat}
\end{subequations}
It is clear that if $(\bff u,\bff v)$ solves \eqref{eqn:weak}, then $\bff u$ solves \eqref{eq:TP}.

\paragraph{A strong version of the first order system.} Most of this section will consist of the analysis of a problem related to \eqref{eqn:weak}, but written in $\mathbb R_+:=[0,\infty)$, using classical time derivatives and vanishing initial conditions. Our data are now smooth enough functions $\bs\beta:[0,\infty)\to \HcurlH$ and $\bs\xi:[0,\infty)\to \HdivH$, and look for 
\begin{subequations}\label{eqn:strong}
\begin{eqnarray} 
\bff u & \in & \mathcal C(\mathbb{R}_+;\bff H(\curl;\mathbb{R}^3))\cap \mathcal C^1(\mathbb{R}_+;\mathbf{L}^2(\mathbb{R}^3\backslash\Gamma)), \label{eq:4.2a}
\\
\bff v & \in &  \mathcal C(\mathbb{R}_+;\bff H(\curl;\mathbb{R}^3\backslash\Gamma)) \cap \mathcal C^1(\mathbb{R}_+;\mathbf{L}^2(\mathbb{R}^3)), \label{eq:4.2b}
\end{eqnarray}
such that for all $t\ge 0$
\begin{alignat}{3}
 \label{eq:4.2c}
& \dot{\bff u}(t) - c\nabla \times \bff v(t)=\bff 0,
&\qquad &
 \dot{\bff v}(t) + c\nabla \times \bff u(t)=\bff 0, \\
 \label{eq:4.2d}
& \pi_\tau \bff u(t)-\bs\beta(t) \in \bs X_h^\circ ,
&\qquad &
\jump{\gamma_\tau \bff v(t)} - \bs \xi (t) \in \bs X_h.
\end{alignat}
and
\begin{equation}
\bff u(0)=\bff 0,\quad \bff v(0)=\bff 0.
\end{equation}
\end{subequations}

\paragraph{An unbounded operator.}
Consider $\mathcal H:=\bff L^2(\mathbb{R}^3\backslash\Gamma )\times \bff L^2(\mathbb{R}^3)$, equipped with its natural norm, and the spaces
\begin{align*}
\bff U_h &:= \{ \bff u\in \bff H(\curl,\mathbb{R}^3) : \pi_\tau \bff u\in \bs X_h^\circ \}, \\
\bff V_h & :=\{ \bff v\in \bff H(\curl,\mathbb{R}^3\backslash\Gamma ) : \jump{\gamma_\tau \bff v}\in \bs X_h\}.
\end{align*}
In the domain $D(\mathcal A)  :=\bff U_h \times \bff V_h$, we define the operator $\mathcal A:D(\mathcal A)\to\mathcal H$ given by
$\mathcal A(\bff u,\bff v):=(c\nabla\times\bff v,-c\nabla\times\bff u)$.

\begin{proposition}\label{prop:4.1}
The operator $\mathcal A:D(\mathcal A)\to \mathcal H$ is the infinitesimal generator of a unitary $C_0$-group of operators in $\mathcal H$.
\end{proposition}
\begin{proof}
According to \cite[Chapter 1, Theorem 4.3]{Pazy:1983} or \cite[Theorem 4.5.4]{Kesavan:1989}, we only need to prove that $\pm \mathcal A$ are maximal dissipative. We first prove that
\[
(\mathcal A W, W)_{\mathcal H} = 0 \quad \forall W \in D(\mathcal A).
\]
For all $W=(\bff u,\bff v)\in \bff U_h\times \bff V_h$, we can write
\begin{align*}
(\mathcal AW,W)_{\mathcal H} &= \left( (c\nabla\times \bff v,-c\nabla \times \bff u),(\bff u,\bff v)\right)_{\mathcal H} \\
&= (c\nabla\times \bff v, \bff u)_{\mathbb{R}^3\backslash\Gamma}-(c\nabla\times \bff u,\bff v)_{\mathbb{R}^3} = -c\langle \jump{\gamma_\tau \bff v},\pi_\tau \bff u\rangle =0,
\end{align*}
since $\jump{\gamma_\tau \bff v}\in \bs X_h$ and $\pi_\tau \bff u\in \bs X_h^\circ$.

Let $\mathcal I:D(\mathcal A)\to \mathcal H$ be the canonical inclusion of $D(\mathcal A)$ into $\mathcal H$. We now prove that $\mathcal I-\mathcal A:D(\mathcal A)\to \mathcal H$ is surjective, i.e., given $\bff f\in \bff L^2(\mathbb{R}^3\backslash\Gamma),\bff g\in \bff L^2(\mathbb{R}^3)$, there exists $(\bff u,\bff v)\in D(\mathcal A)$ satisfying
\begin{equation}\label{eqn:inhomogeneous}
\bff u - c\nabla\times \bff v = \bff f,
\qquad
\bff v + c\nabla\times \bff u = \bff g.
\end{equation}
To prove this, we solve the following coercive variational problem
\begin{equation}
\label{eqn:4.4}
\bff u\in \bff U_h,
\qquad
(\bff u,\bff w)_{\mathbb R^3}+c^2(\nabla\times \bff u,\nabla\times \bff w)_{\mathbb R^3}=
(\bff f,\bff w)_{\mathbb R^3}+c(\bff g,\nabla\times \bff w)_{\mathbb R^3}\quad \forall \bff w\in \bff U_h
\end{equation}
and then define
$
\bff v = \bff g-c\nabla\times \bff u\in \bff L^2(\mathbb{R}^3).
$
We claim that $(\bff u,\bff v) \in D(\mathcal A)$ and that \eqref{eqn:inhomogeneous} is satisfied. To prove this, we first choose an arbitrary $\bff z\in (\mathcal D(\mathbb{R}^3\backslash\Gamma))^3\subset \bff U_h$ as a test function in \eqref{eqn:4.4} and show that
\begin{align*}
\langle \nabla\times \bff v,\bff z\rangle_{(\mathcal{D}'(\mathbb{R}^3\backslash\Gamma))^3\times (\mathcal{D}(\mathbb{R}^3\backslash\Gamma))^3} &= (\bff v,\nabla\times \bff z)_{\mathbb{R}^3} \\
&=(\bff g-c\nabla\times \bff u,\nabla\times \bff z)_{\mathbb{R}^3} = c^{-1}(\bff u-\bff f,\bff z)_{\mathbb{R}^3}.
\end{align*}
This implies $\nabla\times \bff v=c^{-1}(\bff u-\bff f)\in \bff L^2(\mathbb{R}^3\backslash\Gamma)$, and therefore $\bff v\in \bff H(\curl,\mathbb{R}^3\backslash\Gamma)$. For $\bff w\in \bff U_h$, we have
\begin{align*}
-\langle \jump{\gamma_\tau \bff v},\pi_\tau \bff w\rangle &= (\nabla \times \bff v,\bff w)_{\mathbb{R}^3\backslash\Gamma}-(\bff v,\nabla\times \bff w)_{\mathbb{R}^3} \\
&=c^{-1}(\bff u-\bff f,\bff w)_{\mathbb R^3}-(\bff g-c\nabla\times \bff u,\nabla\times \bff w)_{\mathbb R^3}=0.
\end{align*}
The observation $\pi_\tau\bff U_h=\bs X_h^\circ$ leads to
\[
\langle \jump{\gamma_\tau \bff v},\bs \zeta \rangle=0\quad \forall \bs \zeta\in \bs X_h^\circ,
\]
and therefore $\jump{\gamma_\tau \bff v}\in \bs X_h$, which proves that $\bff v\in \bff V_h$. The surjectivity of $\mathcal I-\mathcal A$ and dissipativity imply maximal dissipativity of $\mathcal A$.

To prove the $\mathcal I+\mathcal A$ is surjective, we solve a similar variational problem
\[
\bff u\in \bff U_h,
\qquad
(\bff u,\bff w)_{\mathbb R^3}+c^2(\nabla\times \bff u,\nabla\times \bff w)_{\mathbb R^3}
=(\bff f,\bff w)_{\mathbb R^3}-c(\bff g,\nabla\times \bff w)_{\mathbb R^3}\quad \forall \bff w\in \bff U_h
\]
and then define $ \bff v = \bff g+c\nabla\times \bff u$. The rest of the analysis is a minor variation of the previous case.
\end{proof}

\paragraph{Lifting operator.} The next step is the construction of a lifting operator that will eliminate the non-homogeneous transmission conditions \eqref{eq:4.2d}.

\begin{proposition}\label{prop:4.2}
Given  $\bs\beta \in \HcurlH,\bs\xi\in \HdivH$, there exists a unique pair $(\bff u,\bff v)\in  \bff H(\curl,\mathbb{R}^3)\times \bff H(\curl,\mathbb{R}^3\backslash\Gamma)$ satisfying
\begin{subequations}\label{eq:4.6}
\begin{alignat}{6}
& \bff u = c\nabla\times \bff v, & \qquad & 
\bff v = -c\nabla \times \bff u,\\
& \pi_\tau \bff u-\bs\beta \in \bs X_h^\circ, & &
\jump{\gamma_\tau \bff v} - \bs\xi \in \bs X_h.
\end{alignat}
\end{subequations}
Moreover, there exists $C_\Gamma >0$, independent of the choice of $\bs X_h$, such that
\begin{equation}\label{eq:4.6b}
\| \bff u\|_{\curl,\mathbb{R}^3} + \| \bff v\|_{\curl,\mathbb{R}^3\backslash\Gamma} \leq C_\Gamma \max\{c,c^{-1}\} ( \| \bs\beta\|_{\HcurlG} +\| \bs\xi\|_{\HdivG} )
\end{equation}
\end{proposition}

\begin{proof}
Since $\pi_\tau :\bff H(\curl,\mathbb{R}^3)\to \HcurlH$ is bounded and surjective, we can use a bounded right-inverse to build $\bff{u}^{\bs\beta} \in \bff H(\curl,\mathbb{R}^3)$ such that
\begin{equation}\label{eq:4.7}
\pi_\tau \bff{u}^{\bs\beta} = \bs\beta,
\qquad 
\| \bff{u}^{\bs\beta}\|_{\curl,\mathbb{R}^3} \leq C_1 \| \bs \beta\|_{\HcurlG}.
\end{equation}
Similarly, since $\jump{\gamma_\tau} : \bff H(\curl,\mathbb{R}^3\setminus\Gamma)\to \HdivH$ is bounded and onto, we can choose $\bff{v}^{\bs\xi} \in \bff H(\curl,\mathbb{R}^3\setminus\Gamma)$ satisfying
\begin{equation}\label{eq:4.8}
\jump{\gamma_\tau \bff{v}^{\bs\xi}} = \bs\xi,
\qquad 
\| \bff{v}^{\bs\xi}\|_{\curl,\mathbb{R}^3\setminus\Gamma} \leq C_2 \| \bs \xi\|_{\HdivG}.
\end{equation}
In the proof of Proposition \ref{prop:4.1}, we have shown that $\mathcal I-\mathcal A$ is surjective. Therefore, we can find $(\bff u^0,\bff v^0)\in D(\mathcal A)$ satisfying
\begin{subequations}\label{eq:4.10}
\begin{alignat}{6}
\bff u^0-c\nabla\times\bff v^0 & = \bff f:=c\nabla\times\bff v^{\bs\xi}-\bff u^{\bs\beta},\\
\bff v^0+c\nabla\times \bff u^0 &=\bff g:=-c\nabla\times \bff u^{\bs\beta}-\bff v^{\bs\xi}.
\end{alignat}
\end{subequations}
It is then obvious that $(\bff u,\bff v)=(\bff u^{\bs\beta}+\bff u^0,\bff v^{\bs\xi}+\bff v^0)$ is a solution of \eqref{eq:4.6}. Moreover, proceeding as in the proof of Proposition \ref{prop:4.1}, it is evident that $\bff u^0$ is the unique solution of \eqref{eqn:4.4}. Choosing $\bff w=\bff u^0$ as test function in \eqref{eqn:4.4} and using \eqref{eq:4.7}-\eqref{eq:4.8}, we can bound
\begin{align*}
\min \{1,c\} \| \bff u^0\|_{\curl,\mathbb{R}^3}
&\leq \sqrt{\| \bff u^0\|_{\mathbb{R}^3}^2 + c^2\| \nabla\times \bff u^0\|_{\mathbb{R}^3}^2}
\leq \sqrt{\| \bff f\|_{\mathbb{R}^3\setminus\Gamma}^2 + \| \bff g\|_{\mathbb{R}^3}^2} \\
&\leq \sqrt{2}\max\{1,c\}\max\{C_1,C_2\} (\| \bs \beta\|_{\HcurlG}+\| \bs \xi\|_{\HdivG}).
\end{align*}
Using now \eqref{eq:4.10}, we can bound
\begin{alignat*}{6}
\| \bff v^0\|_{\curl,\mathbb{R}^3\setminus\Gamma} 
& \leq \| \bff v^0\|_{\mathbb{R}^3} + \| \nabla\times \bff v^0\|_{\mathbb{R}^3 \setminus\Gamma}
\leq \| \bff g-c\nabla\times\bff u^0\|_{\mathbb{R}^3} + \| c^{-1}(\bff u^0 -\bff f) \|_{\mathbb{R}^3\setminus\Gamma} \\
&\leq \sqrt{2}\max\{c^{-1},1\}\sqrt{\| \bff u^0\|_{\mathbb{R}^3}^2+c^2 \| \nabla \times \bff u^0\|_{\mathbb{R}^3}^2} + c^{-1}\| \bff f\|_{\mathbb{R}^3\setminus\Gamma}+\| \bff g\|_{\mathbb{R}^3} \\
& \leq 2\max \{c,c^{-1}\} \max \{ C_1,C_2\} (\| \bs \beta\|_{\HcurlG}+\| \bs \xi\|_{\HdivG}) \\
&\hspace{1.5in} +\sqrt{2}\max \{ c,c^{-1}\} C_1 \| \bs\beta\|_{\HcurlG} + \sqrt{2}C_2 \| \bs\xi\|_{\HdivG}.
\end{alignat*}
Hence, the estimate \eqref{eq:4.6b} follows readily. Note that this estimate proves uniqueness of solution of \eqref{eq:4.6}.
\end{proof}

We are now ready to use the previous arguments to prove existence and uniqueness of solution to the non-homogeneous transmission problem \eqref{eqn:strong}. We will need the spaces 
\[
\mathcal C^k_0(\mathbb R_+;X):=\{ f\in \mathcal C^k(\mathbb R_+;X)\,:\, f^{(j)}(0)=0, \quad j\le k-1\}.
\]
and the following auxiliary result from the theory of evolutionary equations on Banach spaces (See \cite[Chapter 4, Corollary 2.5]{Pazy:1983}).

\begin{theorem}
\label{thm:semigroup}
Let $\mathcal A:D(\mathcal A)\subset \mathcal H\to\mathcal H$ be the infinitesimal generator of a contractive $C_0$-semigroup of operators in $\mathcal H$ and let $F\in \mathcal C_0^1(\mathbb R_+;\mathcal H)$. The initial value problem
\[
\dot U(t) = \mathcal A U(t) +F(t)\quad \forall t\geq 0, 
\qquad U(0)=0,
\]
has a unique classical solution
$
U\in \mathcal C^1(\mathbb R_+;\mathcal H)\cap \mathcal C(\mathbb R_+;D(\mathcal A))
$
and we can bound
\[
\| U(t)\|_{\mathcal H} \leq \int_0^t \| F(\tau )\|_{\mathcal H}\,\mathrm{d}\tau, 
\qquad
\| \dot U(t)\|_{\mathcal H} \leq \int_0^t \| \dot F(\tau )\|_{\mathcal H} \,\mathrm{d}\tau, 
\qquad
\forall t\ge 0.
\]
\end{theorem}

\begin{proposition} \label{prop:4.4}
For all $\bs\beta\in \mathcal W_+^2(\HcurlH)$ and $\bs\xi\in \mathcal W_+^2(\HdivH)$, problem \eqref{eqn:strong} with data $(\bs\beta|_{\mathbb{R}_+},\bs\xi|_{\mathbb{R}_+})$ has a unique solution and for all $t\geq 0$,
\begin{align*}
\| \bff u(t)\|_{\curl,\mathbb{R}^3} \leq 4C_\Gamma\max\{c,c^{-2}\} \left( H_2(\bs\beta,t\,|\,\HcurlH)+ H_2(\bs\xi,t\,|\,\HdivH) \right), \\
\| \bff v(t)\|_{\curl,\mathbb{R}^3\backslash\Gamma } \leq 4C_\Gamma\max\{c,c^{-2}\} \left( H_2(\bs\beta,t\,|\,\HcurlH)+ H_2(\bs\xi,t\,|\,\HdivH) \right),
\end{align*}
where $C_\Gamma$ is the constant of Proposition \ref{prop:4.2}.
\end{proposition}
\begin{proof}
We first prove the result with slightly smoother data: $\bs\beta\in \mathcal C_0^2(\mathbb{R}_+;\HcurlH)$ and $\bs\xi\in \mathcal C_0^2(\mathbb{R}_+;\HdivH)$.
Let $L:\HcurlH\times \HdivH\to  \bff H(\curl,\mathbb{R}^3)\times \bff H(\curl,\mathbb{R}^3\backslash\Gamma)$ be the lifting operator defined by Proposition \ref{prop:4.2}. We then define $W_\TC (t) = L( \bs\beta(t),\bs\xi(t))$ for all $t\ge 0$, and consider the function
$F=\dot{W}_\TC - W_\TC=L(\dot{\bs\beta}-\bs\beta,\dot{\bs\xi}-\bs\xi)$. The hypotheses on $\bs\beta$ and $\bs\xi$ imply that $W_\TC \in \mathcal C_0^2(\mathbb{R}_+;\bff H(\curl;\mathbb{R}^3)\times \bff H(\curl;\mathbb{R}^3\backslash\Gamma))$ and therefore $F\in \mathcal C_0^1(\mathbb{R}_+;\mathcal H)$. At this moment, we use Theorem \ref{thm:semigroup} and Proposition \ref{prop:4.1} to define the unique solution $W_0\in \mathcal C^1(\mathbb{R}_+;\mathcal H)\cap \mathcal C(\mathbb{R}_+;D(\mathcal A))$ of
\begin{equation}\label{eq:4.11}
\dot{W}_0(t) = \mathcal A W_0(t) + F(t),\quad t\geq 0,
\qquad
W_0(0) =0.
\end{equation}
It is then easy to show that $(\bff u,\bff v)=W_{\mathrm{TC}}+W_0$ is the unique solution of \eqref{eqn:strong}. (Uniqueness follows from uniqueness of solution to \eqref{eq:4.11}.)

Using Theorem \ref{thm:semigroup}, we estimate
\begin{align*}
\| W_0(t)\|_{\mathcal H} &\leq \int_0^t \| \dot{W}_\TC (\tau )- W_\TC (\tau )\|_{\mathcal H}\,\mathrm{d} \tau \\
& \leq C_\Gamma\max\{c,c^{-1}\} \left( H_1(\bs\beta,t ; \HcurlH) + H_1(\bs\xi,t ; \HdivH) \right).
\end{align*}
Using the bounds for the lifting operator again, we next bound
\begin{align*}
\| W(t) \|_{\mathcal H} &\leq \| W_\TC (t) \|_{\mathcal H} + \| W_0 (t)\|_{\mathcal H} \\
&\leq 2C_\Gamma\max\{c,c^{-1}\} \left( H_1(\bs\beta,t ; \HcurlH) + H_1(\bs\xi,t ; \HdivH) \right).
\end{align*}
Analogously, we obtain
\begin{align*}
\| \dot{W}(t) \|_{\mathcal H} &\leq \| \dot{W}_\TC (t) \|_{\mathcal H} + \| \dot{W}_0 (t)\|_{\mathcal H} \\
&\leq 2C_\Gamma\max\{c,c^{-1}\} \left( H_1(\dot{\bs\beta},t ; \HcurlH) + H_1(\dot{\bs\xi},t ; \HdivH) \right).
\end{align*}
The bounds in the statement of the Proposition follow then from the identities
\[
\| \bff u(t)\|_{\curl,\mathbb{R}^3}^2 = \| \bff u(t)\|_{\mathbb{R}^3\backslash\Gamma}^2 + \| c^{-1}\dot{\bff v}(t)\|_{\mathbb{R}^3}^2, \qquad
\| \bff v(t)\|_{\curl,\mathbb{R}^3\backslash\Gamma}^2 = \| \bff v(t)\|_{\mathbb{R}^3}^2 + \| c^{-1}\dot{\bff u}(t)\|_{\mathbb{R}^3\backslash\Gamma}^2.
\]
To prove the result for the weaker hypotheses in the statement we just need to use a simple density argument. Alternatively, we can use a variant of Theorem \ref{thm:semigroup} where $F\in \mathcal C(\mathbb R_+;\mathcal H)$, $F(0)=0$ and $\dot F$ is integrable.
\end{proof}


\paragraph{Extension operator.} Consider a continuous function $f:[0,\infty)\to X$ and
\[ 
(Ef)(t):= \begin{cases}
f(t), & t\geq 0,\\
0,& t<0, \end{cases} 
\]
which is a causal $X$-valued distribution. If $\|f(t)\|\le C (1+t^m)$, then $Ef\in \TD(X)$. 

\begin{proposition} \label{prop:4.6}
Let $\bs\beta\in \mathcal W_+^2(\HcurlH),\bs\xi\in \mathcal W_+^2(\HdivH)$ and let $(\bff u,\bff v)$ be the unique solution of \eqref{eqn:strong} with data $(\bs\beta |_{\mathbb{R}_+},\bs\xi |_{\mathbb{R}_+})$. Then $(E\bff u,E\bff v)$ is the solution of \eqref{eqn:weak} with $(\bs \beta,\bs \xi)$ as data. Therefore, $E\bff u=\mathcal G_h*\bs\beta+\mathcal E_h*\bs\xi$.
\end{proposition}

\begin{proof}
We first note that the hypotheses on $\bs\beta$ and $\bs\xi$ imply that 
\[
H_2(\bs\beta,t\,|\,\HcurlH)+ H_2(\bs\xi,t\,|\,\HdivH)\le C(1+t^2)
\]
for some $C>0$. 
It is now clear from the above and Proposition \ref{prop:4.4} that $E\bff v\in \TD (\bff H(\curl;\mathbb{R}^3\backslash\Gamma))$. It follows from rom \eqref{eq:4.2a} that
\[
\partial^{-1} \bff u \in \mathcal C(\mathbb{R}_+;\bff H(\curl^2;\mathbb{R}^3\backslash\Gamma)) \cap \mathcal C^1(\mathbb{R}_+;\bff H(\curl;\mathbb{R}^3))\cap \mathcal C^2(\mathbb{R}_+;\mathbf{L}^2(\mathbb{R}^3\backslash\Gamma)).
\]
The fact that there exists some other $C>0$ such that
\[
\| \partial^{-1} \bff u(t)\|_{\bff M} \leq \int_0^t \| \bff u(\tau )\|_{\curl,\mathbb{R}^3}\,\mathrm{d}\tau + \| c^{-1} \bff v(t )\|_{\curl,\mathbb{R}^3} \leq C(1+t^3)
\]
implies $E\partial^{-1} \bff u\in \TD (\bff M)$ and therefore $E\bff u\in \TD (\bff M)$.
Since $\bff u(0)=\bff 0$ and $\bff v(0)=\bff 0$, we have
\begin{align*}
{d\over dt}(E\bff u) &=E\dot{\bff u}=E(c\nabla\times \bff v)=\nabla\times (cE\bff v), \\
{d\over dt}(E\bff v) &=E\dot{\bff v}=-E(c\nabla\times \bff u)=-\nabla\times (cE\bff u).
\end{align*}
The transmission boundary conditions can be interpreted in the distributional sense naturally.
\end{proof}

\begin{proposition} \label{prop:4.7}
Let $\bs\beta\in \mathcal W_+^2(\HcurlH),\bs\xi\in \mathcal W_+^2(\HdivH)$ and $\bff u$ be the solution to \eqref{eq:TP},i.e., $\bff u=\mathcal G_h * \bs \beta+\mathcal E_h * \bs\xi$. Then $\bff u\in \mathcal C_+^0(\bff H(\curl;\mathbb{R}^3))\cap \mathcal C_+^1(\bff L^2(\mathbb{R}^3\setminus\Gamma)),\partial^{-1}\nabla\times \bff u\in \mathcal C_+^0(\bff H(\curl;\mathbb{R}^3\backslash\Gamma)),\partial^{-1} \jump{\bff u}_N\in \mathcal C_+^0(\HdivH)$ and there exists a constant $C_\Gamma'>0$ independent of the choice of $\bs X_h$ such that
\begin{align*}
\| & \bff u(t)\|_{\curl,\mathbb{R}^3} + \| c\,\partial^{-1} \nabla \times \bff u(t)\|_{\curl,\mathbb{R}^3\backslash\Gamma } + \| c\,\partial^{-1} \jump{\bff u}_N (t)\|_{\HdivG }\\
&\leq C_\Gamma'\max \{c,c^{-2}\} \left( H_2(\bs\beta,t\,|\,\HcurlH)+ H_2(\bs\xi,t\,|\,\HdivH) \right) \quad \forall t\geq 0.
\end{align*}
\end{proposition}
\begin{proof}
Defining $\bff v=-c\,\partial^{-1}\nabla\times \bff u$, $(\bff u,\bff v)$ is the solution to \eqref{eqn:weak}. By Proposition \ref{prop:4.6} and uniqueness of solution of \eqref{eqn:weak}, $(\bff u|_{\mathbb{R}_+},\bff v|_{\mathbb{R}_+})$ is the solution to \eqref{eqn:strong} with data $(\bs\beta|_{\mathbb{R}_+},\bs\xi|_{\mathbb{R}_+})$ and we have estimates as in Proposition \ref{prop:4.4}. The bound for $\partial^{-1}\nabla\times \bff u$ follows from $\bff v=-c\,\partial^{-1} \nabla\times \bff u$. The bound for $\partial^{-1}\jump{\bff u}_N$ results from the boundedness of the operator $\jump{\gamma_\tau}:\bff H(\curl,\mathbb{R}^3\setminus\Gamma)\to \HcurlH$.
\end{proof}

\paragraph{Proofs of all the results of Section \ref{sec:2}.}
By Proposition \ref{prop:3.3}, $\bff E^h$ and $\bff J^h$ in Theorem \ref{th:1} can be represented as $\bff E^h =\mathcal G_h * \bs\beta, \bff J^h = -c\,\partial^{-1}\jump{\bff E^h}$. Their bounds and regularity are a consequence of Proposition \ref{prop:4.7}  with $\bs\xi=\bff 0$. Corollary \ref{cor:2.3} is a special case of Theorem \ref{th:1} when $\bs X_h^\circ = \{ \bff 0\}$. 
Similarly, the bounds for $\bff E-\bff E^h =\mathcal E_h* \bff J$ and $\bff J-\bff J^h =-c\,\partial^{-1}\jump{\bff E-\bff E^h}$ in Theorem \ref{th:2} follows from Proposition \ref{prop:4.7} with $\bs\beta=\bff 0$ and Proposition \ref{prop:3.4}. Corollary \ref{cor:2.4} is a special case of Theorem \ref{th:2} when $\bs X_h = \{ \bff 0 \}$.

\section{Comparison with existing results}

In this short section we compare our results with those of \cite{BaBaSaVe:2013} (which develop the techniques in \cite{Terrasse:1993}) and \cite{Rynne:1999}.

The following two results are consequences of estimates in  \cite{BaBaSaVe:2013}. On the one hand, \cite[Theorem 4.4(a)]{BaBaSaVe:2013} is a frequency domain result about $\mathcal V^{-1}$(also in \cite[Lemma 2]{LiMoWe:2015}). We convert it to a time domain result using \cite[Theorem 7.1]{DoSa:2013}.

\begin{theorem} \label{prop:5.1}
If we solve $\bs{\mathcal V}_c*\bff J = \bs\beta$ when $\bs\beta \in \mathcal W_+^4(\HcurlH)$, then \\ $\bff J\in \mathcal C_+^0(\HdivH)$ and there is a constant $C$ independent of $t$ such that for arbitrary $t\geq 0$,
\[
\| \bff J(t) \|_{\HdivG} \leq C{t^2\over 1+t}\max\{ 1,t^2\} \int_0^t \| \bs\beta^{(4)} (\tau)\|_{\HdivG} \,\mathrm{d}\tau.
\]
\end{theorem}

The representation formula for the error of Galerkin semidiscretization \cite[Formula (43)]{BaBaSaVe:2013}, the discrete stability estimate \cite[Lemma 4.8]{BaBaSaVe:2013} and the boundary integral operator bound \cite[Formula (32)]{BaBaSaVe:2013} can be used to provide a time domain bound \cite[Formula (44)]{BaBaSaVe:2013} without explicitly displaying the constant's dependence on time. With the same estimates and  \cite[Theorem 7.1]{DoSa:2013} we can prove the following theorem, which admits a better comparison with our results.

\begin{theorem} \label{prop:5.2}
Assume $\Omega_-$ is a polyhedron. Choose $\bs X_h$ to be the lowest order Raviart-Thomas elements on a surface mesh over the polyhedron. Let $\bff J$ and $\bff J^h$ be the solutions to \eqref{eq:2.4} and \eqref{eq:2.6}. If $\bff J\in \mathcal W_+^6(\HdivH)$, then there is a constant $C$ independent of $t$ such that for all $t\geq 0$,
\[
\| \bff J(t)-\bff J^h(t)\|_{\HdivG} \leq C{t^3\over 1+t}\max \{1,t^4\} \int_0^t \| (\bff J-\Pi_h \bff J)^{(6)}(\tau )\|_{\HdivG}\,\mathrm{d}\tau.
\]
\end{theorem}
Compared to Corollary \ref{cor:2.3} and Theorem \ref{th:2}, these estimates require higher temporal regularity of the input data. As mentioned in the introduction, there is a notable loss when going back and forth through the Laplace domain. Theorem \ref{prop:5.2} is written for a polyhedral domain, but the techniques in \cite{BaBaSaVe:2013} can be easily extended to any Lipschitz domain and any discrete subspace defined on it. The constant grows polynomially in time and henceforth the result does not rule out long time polynomial growth of the error even when the input data is compactly supported, which we show not to happen.

The following result is a special case of \cite[Theorem 3.1]{Rynne:1999} when $p=1$. It ruled out the long time growth given compactly supported input data. Nevertheless, understanding the underlying theory it relies on presents a daunting challenge for anyone who intends to extend the result to a more general setting. For example, the approach of \cite{Rynne:1999} gives few clues to a possible discrete error estimate. In comparison with Theorem \ref{th:1}, the estimate applies to scatterers only with smooth boundaries, while our results cover the scatterers with general Lipschitz boundary. Our result also tolerates input data with rougher spatial regularity, $\HcurlH$ versus $\bff H^{3/2}(\Gamma)$.
\begin{theorem}
Assume $\Omega_-$ is a bounded domain with $C^\infty$ boundary. Given \\ $\bs\beta\in \mathcal C_+^0(\bff H^{3/2}(\Gamma))\cap \mathcal C_+^1(\bff H^{1/2}(\Gamma))$, compactly supported in $[0,T]$, let $\bff J$ be the solution to $\bs{\mathcal V}_c*\bff J = \bs\beta$ and define $\bff E=\bs{\mathcal S}_c*\bff J$, then $\bff E\in \mathcal C_+^0(\bff H^1(\Omega_+))\cap \mathcal C_+^1(\bff L^2(\Omega_+)),\bff J\in \mathcal C_+^0(\bff H^{1/2}(\Gamma))$ and the following bound holds for all $t\geq 0$,
\[
\| \bff E(t)\|_{\bff H^1(\Omega_+)} + \| \bff J(t)\|_{\bff H^{1/2}(\Gamma)} \leq C_T \sup_{\tau\in [0,T]} \left( \| \dot{\bs\beta}(\tau)\|_{\bff H^{1/2}(\Gamma)} + \| \bs\beta(\tau)\|_{\bff H^{3/2}(\Gamma )} \right).
\]
\end{theorem}

\bibliographystyle{abbrv}
\bibliography{biblio}

\end{document}